\documentclass[reqno, 12pt,a4letter]{amsart}
\usepackage{amsmath,amsxtra,amssymb,latexsym, amscd,amsthm}
\usepackage[utf8]{inputenc}
\usepackage[mathscr]{euscript}
\usepackage{mathrsfs}
\usepackage[english]{babel}
\usepackage{color}
\usepackage{bbding}
\newtheorem {theorem}{Theorem}[section]

\newtheorem {lemma}{Lemma}[section]
\newtheorem {example}{Example}[section]
\newtheorem {definition}{Definition}[section]
\newtheorem {remark}{Remark}[section]

\def\ar{a\kern-.370em\raise.16ex\hbox{\char95\kern-0.53ex\char'47}\kern.05em}
\def\ees{{\accent"5E e}\kern-.385em\raise.2ex\hbox{\char'23}\kern-.08em}
\def\eex{{\accent"5E e}\kern-.470em\raise.3ex\hbox{\char'176}}
\def\AR{A\kern-.46em\raise.80ex\hbox{\char95\kern-0.53ex\char'47}\kern.13em}
\def\EES{{\accent"5E E}\kern-.5em\raise.8ex\hbox{\char'23 }}
\def\EEX{{\accent"5E E}\kern-.60em\raise.9ex\hbox{\char'176}\kern.1em}
\def\ow{o\kern-.42em\raise.82ex\hbox{
  \vrule width .12em height .0ex depth .075ex \kern-0.16em \char'56}\kern-.07em}
\def\OW{O\kern-.460em\raise1.36ex\hbox{
\vrule width .13em height .0ex depth .075ex \kern-0.16em \char'56}\kern-.07em}
\def\UW{U\kern-.42em\raise1.36ex\hbox{
\vrule width .13em height .0ex depth .075ex \kern-0.16em \char'56}\kern-.07em}
\def\DD{D\kern-.7em\raise0.4ex\hbox{\char '55}\kern.33em}

\title[]{COMPACTNESS CRITERIA FOR REAL ALGEBRAIC SETS AND NEWTON POLYHEDRA}

\author{PH\'{U}-PH\'{A}T PH\d{A}M$^1$}
\address{$^{1}$Department of Mathematics, University of Dalat, 1 Phu Dong Thien Vuong, Dalat, Vietnam}
\email{phatpham.pr13@gmail.com}

\author{TI\EES N-S\OW n PH\d{A}M$^2$}
\address{$^2$Department of Mathematics, University of Dalat, 1 Phu Dong Thien Vuong, Dalat, Vietnam}
\email{sonpt@dlu.edu.vn}



\thanks{The authors were partially supported by Vietnam National Foundation for Science and Technology Development (NAFOSTED), grant 101.04-2016.05}

\date{ \today}

\begin{document}
\maketitle

\begin{abstract}
Let $f \colon \mathbb{R}^n \rightarrow \mathbb{R}$ be a polynomial and $\mathcal{Z}(f)$ its zero set. In this paper, in terms of the so-called Newton polyhedron of $f,$ we present
a necessary criterion and a sufficient condition for the compactness of $\mathcal{Z}(f).$ From this we derive necessary and sufficient criteria for the stable compactness of  $\mathcal{Z}(f).$
\end{abstract}

\section{Introduction}

We do not, at present, have a complete understanding of the possible topologies of real algebraic sets of given degree. For any given real affine plane algebraic curve the problem is much easier, but is still complicated in the general case. A survey of the current state of knowledge and some new results in the case of plane curves, may be found in de la Puente~\cite{Puente2002}.

In this paper, we are interested in the compactness and the stable compactness of real algebraic sets. More precisely, let $f \colon \mathbb{R}^n \rightarrow \mathbb{R}$ be a nonconstant polynomial and $\mathcal{Z}(f)$ its zero set. We would like to know (i) when the set $\mathcal{Z}(f)$ is compact, and (ii) when the set $\mathcal{Z}(f)$ is stably compact in the sense that it remains compact for all sufficiently small perturbations of the coefficients of the polynomial $f.$

In the univariate case, it is easy to see that $\mathcal{Z}(f)$ is a finite set and is stably compact.

In the two-dimensional case (i.e., $n = 2)$, Stalker~\cite{Stalker2007} provides a necessary criterion and a sufficient condition for the compactness of $\mathcal{Z}(f),$ both of which can be stated in terms of the Newton polyhedron of the polynomial $f.$ However, his clever argument is not easy to extend to the higher dimension case.

On the other hand, Marshall \cite[Theorem~5.1]{Marshall2003} gives a necessary and sufficient condition for the stable compactness of  
sets described by polynomial inequalities in terms of homogeneous components of highest degrees of the defining polynomials.

Inspired by the above cited works, assume that $n \ge 2,$ we present two sets of conditions for the compactness of $\mathcal{Z}(f),$ one necessary and one sufficient. From this we derive necessary and sufficient criteria for the stable compactness of  $\mathcal{Z}(f).$ All these conditions are characterized in terms of the Newton polyhedron of the polynomial $f.$

The paper is structured as follows.  In Section~\ref{Preliminary}, we recall some notations and definitions which are used throughout this paper. 
The results and their proofs are given in Section~\ref{Results}.

\section{Notations and definitions} \label{Preliminary}

Throughout the text, we suppose $n \ge 2$ and abbreviate $(x_1, \ldots, x_n)$ by $x.$ For each subset $J \subset \{1, \ldots, n\},$ we define
$$\mathbb{R}^J := \{x \in \mathbb{R}^n \ | \ x_j = 0 \textrm{ for all } j \not \in J\}.$$
We denote by $\mathbb{Z}_+$ the set of non-negative integer numbers. If $\alpha = (\alpha_1, \ldots, \alpha_n) \in \mathbb{Z}_+^n,$ we will denote by $x^\alpha$ the monomial 
$x_1^{\alpha_1} \cdots x_n^{\alpha_n}$ and by $| \alpha|$ the sum $\alpha_1 + \cdots + \alpha_n.$ 

Let $f \colon \mathbb{R}^n \to \mathbb{R}$ be a polynomial function. Suppose that $f$ is written as $f = \sum_{\alpha} a_\alpha x^\alpha.$ 
The {\em norm} of $f$ is defined to be $\Vert f \Vert := \max_{\alpha} |a_{\alpha}|.$ 
The {\em Newton polyhedron (at infinity)} of $f$, denoted by $\Gamma(f),$ is defined as the convex hull in $\mathbb{R}^n$ of the set $\{\alpha \ | \ a_\alpha \ne 0\}.$ If $f \equiv 0,$ then we set $\Gamma(f) = \emptyset.$ 

Given a nonzero vector $q \in \mathbb{R}^n,$ we define
\begin{eqnarray*}
d(q, \Gamma(f)) &:=& \min \{\langle q, \alpha \rangle \ | \ \alpha \in \Gamma(f)\}, \\
\Delta(q, \Gamma(f)) &:=& \{\alpha \in \Gamma(f) \ | \ \langle q, \alpha \rangle = d(q, \Gamma(f)) \}.
\end{eqnarray*}
By definition, for each nonzero vector $q \in \mathbb{R}^n,$ $\Delta(q, \Gamma(f)) $ is a face of $\Gamma(f).$ Conversely, if $\Delta$ is a face of $\Gamma(f)$ then there exists a nonzero vector $q \in \mathbb{R}^n$ such that $\Delta = \Delta(q, \Gamma(f)).$  The {\em Newton boundary (at infinity)} of $f$, denoted by $\Gamma_{\infty}(f),$ is defined as the union of all faces
$\Delta(q, \Gamma(f))$ for some $q \in \mathbb{R}^n$ with $\min_{j = 1, \ldots, n}q_j < 0.$ For each face $\Delta$ of $\Gamma(f),$  we define $f_\Delta$ to be the polynomial $\sum_{\alpha \in \Delta}a_\alpha x^\alpha.$

\begin{definition}[see \cite{Kouchnirenko1976}]{\rm
We say that $f$ is {\em non-degenerate (at infinity)} if, and only if, for each face $\Delta \in \Gamma_{\infty} (f),$ the system of equations
$$f_\Delta(x) = \frac{\partial f_{\Delta}}{\partial x_1}(x)  =  \cdots =  \frac{\partial f_{\Delta}}{\partial x_n}(x) = 0$$
has no solution in $(\mathbb{R}\setminus \{0\})^n.$ 
}\end{definition}

\begin{remark}{\rm
It is worth mentioning that the class of polynomial functions (with fixed Newton polyhedra), which are non-degenerate at infinity, is an open and dense semi-algebraic set  in the corresponding Euclidean space of data; see \cite[Theorem~5.2]{HaHV2017}.
}\end{remark}

\section{Results and proofs}\label{Results}

From now on let $f \colon \mathbb{R}^n \to \mathbb{R}$ be a nonconstant polynomial in $n \ge 2$ variables and let $\mathcal{Z}(f)$ be its zero set:
$$\mathcal{Z}(f) := \{x \in \mathbb{R}^n \ | \ f(x) = 0\}.$$
We start with the following simple observation.
\begin{lemma} \label{Lemma31}
If $\mathcal{Z}(f)$ is compact, then $f$ is bounded either from below or from above.
\end{lemma}

\begin{proof}
Suppose the assertion of the lemma is false. We have
$$\lim_{R \to +\infty} \min_{\Vert x\Vert = R} f(x)  = -\infty \quad \textrm{   and   } \quad \lim_{R \to +\infty}  \max_{\Vert x\Vert = R}f(x)  =  +\infty.$$
Then, for $R$ sufficiently large, there exist $a, b \in \mathbb{R}^n$ with $\Vert a \Vert = \Vert b \Vert = R$ such that $f(a) < 0 < f(b).$ 
Furthermore, since $\mathcal{Z}(f)$ is compact, we may assume that 
$$\mathcal{Z}(f) \subset \{ x \in \mathbb{R}^n \ | \ \|x\| < R\},$$
after perhaps increasing $R.$

On the other hand, the sphere $\mathbb{S}^{n-1}_R := \{x\in \mathbb{R}^n \ | \  \Vert x \Vert\ =R\} $ is path-connected (note that $n \ge 2).$ Hence, there is a continuous curve
$$\phi \colon [0, 1] \rightarrow \mathbb{S}^{n-1}_R,\ \  t \mapsto \phi(t),$$
such that $\phi(0) = a $ and $\phi(1) = b.$ Consequently, the composition function $f \circ \phi \colon [0, 1] \rightarrow \mathbb{R}$ is continuous and satisfies 
$$(f \circ \phi) (0)  \times (f \circ \phi) (1)  =  f(a) \times f(b) < 0.$$
Thanks to the mean value theorem, we can find $t_0 \in (0,1)$ such that $(f \circ \phi)(t_0) = 0,$ a contradiction.
\end{proof}

\begin{lemma} \label{Lemma32}
Assume that $f$ is bounded from below and its zero set $\mathcal{Z}(f)$ is compact. Then, the sub-level set $\{x \in \mathbb{R}^n \ | \ f(x)\leq 0\}$ is compact. 
\end{lemma}
\begin{proof}
On the contrary, suppose that $\{x \in \mathbb{R}^n \ | \ f(x)\leq 0\}$ is not compact. Then there exists a sequence $\{a^k\}_{k \ge 1} \subset \mathbb{R}^n$ such that
$$\lim_{k \to \infty} \|a^k \| = + \infty \quad \textrm{ and } \quad f(a^k) \le 0 \ \textrm{ for all } k.$$
Let $b^k$ be an optimal solution of the problem
$$\max_{\|x\| \ = \ \|a^k\|} f(x).$$
Since $f$ is bounded from below, it cannot bounded from above. In particular, 
$$\lim_{k \to \infty} f(b^k) = + \infty.$$
Therefore,  for all $k$ sufficiently large,
$$f(a^k) \times f(b^k) \le 0.$$
As in the proof of Lemma~\ref{Lemma31}, we can find $c^k \in \mathbb{R}^n$ with $\|c^k\| = \|a^k\| = \|b^k\|$ such that
$f(c^k) = 0,$ which contradicts the compactness of $\mathcal{Z}(f).$
\end{proof}

The following is a necessary criterion for compactness of real algebraic sets. 

\begin{theorem}[{Compare \cite[Theorem~2] {Stalker2007}}] \label{Theorem1}
Suppose that $\mathcal{Z}(f)$ is compact. Then
\begin{enumerate}
	\item[(i)]  $f|_{\mathbb{R}^J} \not \equiv 0$ for all $J \subset \{1,\ldots,n\}$,
	
	\item[(ii)] One of the following statements holds
	\begin{enumerate}
	\item[(ii1)] $f$ is bounded from below and $f_{\Delta} \ge 0$ on $\mathbb{R}^n$ for all $\Delta \in \Gamma_{\infty}(f).$
	\item[(ii2)] $f$ is bounded from above and $f_{\Delta} \le 0$ on $\mathbb{R}^n$ for all $\Delta \in \Gamma_{\infty}(f).$
\end{enumerate}
\end{enumerate}
\end{theorem}

\begin{proof}
(i) This is obvious.

(ii) By Lemma~\ref{Lemma31}, $f$ is bounded either from below or from above. 

Assume that $f$ is bounded from below; the case $f$ is bounded from above is treated similarly. 
Take any $\Delta \in \Gamma_{\infty}(f).$ We will show that $f_\Delta \geq 0$ on $\mathbb{R}^n$. In fact, since $f$ is continuous, it suffices to prove that $f_\Delta \geq 0$ on $(\mathbb{R} \setminus \{0\})^n$. Suppose to the contrary, there is a point $x^0 \in (\mathbb{R}\setminus \{0\})^n$ such that $f_{\Delta}(x^0) < 0.$ By definition, there exists a vector $q \in \mathbb{R}^n$ with $\min_{j = 1, \ldots, n} q_j  <0$ such that $\Delta = \Delta(q, \Gamma(f)).$ Define the monomial curve 
$$\phi \colon (0, + \infty) \rightarrow \mathbb{R}^n, \quad t \mapsto (x^0_1 t^{q_1}, \ldots, x^0_n t^{q_n}).$$
Then $\Vert \phi(t) \Vert\rightarrow +\infty$ as $t\rightarrow 0^+.$ Furthermore, a simple calculation shows that for all $t > 0$ small enough,
$$f(\phi(t)) = f_{\Delta}(x^0)t^d + \textrm{ higher-order terms in }t,$$
where $d := d(q, \Gamma(f)).$ Since $f_\Delta(x^0) < 0$ it follows that $f(\phi(t)) < 0$ for all $t > 0$ sufficiently small. Hence, 
the sub-level set $\{x \in \mathbb{R}^n \ | \ f(x) \leq 0\}$ is not compact, which contradicts Lemma~\ref{Lemma32}.
\end{proof}

The following example shows that the converse of Theorem~\ref{Theorem1} does not hold.
\begin{example}{\rm 
Let $n = 2$ and consider the polynomial
$$f(x_1,x_2) := (x_1-x_2)^2.$$
By definition, the Newton polyhedron $\Gamma(f)$ is a segment joining the two points $(2, 0)$ and $(0, 2),$ and so
the Newton boundary $\Gamma_\infty(f)$ is the union of the faces:
$$\Delta_1 := \{(2, 0)\}, \quad \Delta_2 := \{(0, 2)\}, \quad \textrm{ and } \quad \Delta_3 := \{(1 - t) (2, 0) + t (0, 2) \ | \ 0 \le t \le 1\}.$$
Clearly, the polynomials $f_{\Delta_1}(x_1,x_2) = x_1^2, f_{\Delta_2}(x_1,x_2) = x_2^2,$  and $f_{\Delta_3}(x_1,x_2) = (x_1-x_2)^2$ are all non-negative on $\mathbb{R}^n$. However, $\mathcal{Z}(f) = \{(x_1,x_2)\in \mathbb{R}^2 \ | \ x_1=x_2\}$ is not compact. On the other hand, we have the following statement.
}\end{example}

The following is a sufficient condition for compactness of real algebraic sets. 

\begin{theorem}[{Compare \cite[Theorem~1] {Stalker2007}}] \label{Theorem2}
Suppose that
\begin{enumerate}
	\item[(i)]  $f|_{\mathbb{R}^J} \not\equiv 0$ for all $J \subset \{1,\ldots,n\}$,
	\item[(ii)] One of the following statements holds
	\begin{enumerate}
	\item[(ii1)] $f_{\Delta} > 0$ on $(\mathbb{R}\setminus \{0\})^n$ for all $\Delta \in \Gamma_{\infty}(f)$.
	\item[(ii2)] $f_{\Delta} < 0$ on $(\mathbb{R}\setminus \{0\})^n$ for all $\Delta \in \Gamma_{\infty}(f)$.
\end{enumerate}
\end{enumerate}
Then $\mathcal{Z}(f)$ is compact. 
\end{theorem}

\begin{proof}
Assume the assertion of the theorem is false, i.e.,  there exists a sequence of points $\{a^k\}_{k \ge 1} \subset \mathcal{Z}(f)$ 
such that $\lim_{k \to \infty} \Vert a^k \Vert =  +\infty.$ By the Curve Selection Lemma at infinity (see \cite[Theorem~1.12]{HaHV2017}, \cite{Milnor1968}), there exists an analytic curve 
$$\phi \colon (0,\epsilon) \rightarrow \mathbb{R}^n, \quad t \mapsto (\phi_1(t), \ldots, \phi_n(t)),$$ 
such that
\begin{enumerate}
\item[(a)]  $\Vert \phi(t) \Vert \rightarrow +\infty$ as $t \rightarrow 0^+,$
\item[(b)]  $f(\phi(t))=0$ for $t \in (0,\epsilon).$
\end{enumerate}
Let $J := \{j \ | \  \phi_j \not \equiv 0\} \subseteq \{1,\ldots,n\}$. By Condition~(a), $J \neq \emptyset$. For $ j \in J$, we can expand the function $\phi_j$ in terms of parameter, say
 $$\phi_j(t) = x^0_j t^{q_j} + \textrm{ higher-order terms in } t ,$$
where $x^0_j \ne 0$ and $q_j \in \mathbb{Q}.$ By Condition~(a) again, then $\min_{j \in J}q_j <0.$ 

Note that the curve $\phi$ lies in $\mathbb{R}^J \cap \mathcal{Z}(f).$ Hence, by the assumption (i), the restriction of $f$ on $\mathbb{R}^J$ is not constant; in particular, the polyhedron 
$\Gamma(f|_{\mathbb{R}^J})$ is nonempty and different from $\{0\}.$ Let $d$ be the minimal value of the linear function $\sum_{j \in J}q_j \alpha_j$ on $\Gamma(f|_{\mathbb{R}^J})$ and let $\Delta$ be the maximal face of $\Gamma(f|_{\mathbb{R}^J})$ (maximal with respect to the inclusion of faces) where the linear function takes this value, i.e.,
\begin{eqnarray*}
d &:=& d(q, \Gamma(f|_{\mathbb{R}^J})) \quad \textrm{ and } \quad \Delta \ := \ \Delta(q, \Gamma(f|_{\mathbb{R}^J})).
\end{eqnarray*}
(Here we put $q_j := 0$ for $j \not \in J.)$ Then $\Delta \in \Gamma_{\infty}(f)$ because $\min_{i\in J}q_j < 0.$ Furthermore, we have asymptotically as $t \to 0^+,$
$$f(\phi(t)) = f_{\Delta}(x^0)t^d + \textrm{ higher-order terms in }t,$$
where $x^0 := (x^0_1,\ldots, x^0_n)$ with $x^0_j := 1$ for $j \notin J.$ (Note that the polynomial $f_\Delta$ does not depend on $x_j$ for $j \not \in J.)$ Combining this with Condition~(b) gives $f_{\Delta}(x^0) = 0,$ which contradicts the assumption~(ii).
\end{proof}

In the rest of this paper we study the stable compactness of real algebraic sets, which is easier to check than compactness. 

\begin{definition}{\rm 
The set $\mathcal{Z}(f)$ is called {\em stably compact} if there is $\epsilon >0$ such that  $\mathcal{Z}(f + g)$ is compact for all polynomials $g \colon \mathbb{R}^n \rightarrow \mathbb{R}$ with $\Gamma(g) \subseteq \Gamma(f)$ and $\Vert g\Vert <\epsilon.$
}\end{definition}

By definition, the set $\mathcal{Z}(f)$ is stably compact iff it remains compact for all sufficiently small perturbations of the coefficients of the polynomial $f.$
 
\begin{lemma} \label{Lemma33}
The following conditions are equivalent:
\begin{enumerate}
	\item[(i)]  $f_{\Delta} \ne 0$ on $(\mathbb{R}\setminus \{0\})^n$ for all $\Delta \in \Gamma_{\infty}(f)$.
	\item[(ii)] One of the following statements holds
	\begin{enumerate}
	\item[(ii1)] $f_{\Delta} > 0$ on $(\mathbb{R}\setminus \{0\})^n$ for all $\Delta \in \Gamma_{\infty}(f)$.
	\item[(ii2)] $f_{\Delta} < 0$ on $(\mathbb{R}\setminus \{0\})^n$ for all $\Delta \in \Gamma_{\infty}(f)$.
\end{enumerate}
\end{enumerate}
\end{lemma}

\begin{proof}
It suffices to show  the implication (i) $\Rightarrow$ (ii). Assume this is not the case, it means that there exist faces $\Delta_1, \Delta_2 \in \Gamma_{\infty}(f)$ such that $f_{\Delta_1} > 0 > f_{\Delta_2}$ on $(\mathbb{R}\setminus \{0\})^n.$ We may assume further that 
these faces are adjacent, i.e., $\Delta := \Delta_1 \cap \Delta_2 \ne \emptyset.$ Then $\Delta \in \Gamma_{\infty}(f).$ By assumption, $f_{\Delta} \ne 0$ on 
$(\mathbb{R}\setminus \{0\})^n.$ Fix $x^0 := (x^0_1, \ldots, x^0_n)  \in (\mathbb{R}\setminus \{0\})^n,$ and without loss of generality, we may assume that $f_{\Delta}(x^0)  > 0.$ 
By definition, there exists a vector $q$ with $\min_{j = 1, \ldots, n} q_j < 0$ such that $\Delta = \Delta(q, \Gamma(f)).$ A simple calculation shows that
\begin{eqnarray*}
f_{\Delta_2} (t^{q_1} x^0_1, \ldots, t^{q_n} x^0_n) &=& t^d f_{\Delta} (x^0)  + \textrm{ higher-order terms in }t,
\end{eqnarray*}
where $d := d(q, \Gamma(f)).$ Since $f_{\Delta} (x^0) > 0,$ this implies that $f_{\Delta_2} (t^{q_1} x^0_1, \ldots, t^{q_n} x^0_n) > 0$ for all $t > 0$ small enough, which contradicts the fact that $f_{\Delta_2} < 0$ on  $(\mathbb{R}\setminus \{0\})^n.$
\end{proof}

In what follows, let $\mathscr{P}(x) := \sum_{\alpha \in \Gamma(f) \cap \mathbb{Z}_+^n}|x ^{\alpha}|$ and for each face $\Delta$ of the polyhedron $\Gamma(f),$ set $\mathscr{P}_{\Delta}(x) := \sum_{\alpha \in \Delta \cap \mathbb{Z}_+^n}|x ^{\alpha}|.$ By definition, the functions $\mathscr{P}$ and $\mathscr{P}_{\Delta}$ are positive on $(\mathbb{R}\setminus \{0\})^n.$ 

\begin{remark}{\rm
Let $\widetilde{\mathscr{P}}(x) := \sum_{\alpha}|x ^{\alpha}|,$ where the sum is taken over all the vertices of $\Gamma(f).$ Then there exist positive constants $c_1, c_2,$ and $R$ such that
\begin{eqnarray*}
c_1 \mathscr{P}(x) & \le & \widetilde{\mathscr{P}}(x) \ \le \ c_2 \mathscr{P}(x) \quad \textrm{ for all } \quad x \in \mathbb{R}^n.
\end{eqnarray*}
Indeed, the right-hand inequality clearly holds with $c_2 := 1.$ To see the left-hand inequality, let $v^1, \ldots, v^s$ be the vertices of the polyhedron $\Gamma(f)$. Then, for each $\alpha \in \Gamma(f)$, there exist non-negative real numbers $\lambda_1, \ldots, \lambda_s,$ with 
$\lambda_1 + \cdots + \lambda_s = 1,$ such that 
$$\alpha = \lambda_1v^1 + \cdots + \lambda_sv^s.$$
Consequently, we have for all $x \in \mathbb{R}^n,$
\begin{eqnarray*}
|x^\alpha| = |x^{\lambda_1v^1 + \cdots + \lambda_sv^s}| 
&=& (|x^{v^1}|)^{\lambda_1} \cdots(|x^{v^s}|)^{\lambda_s}\\
&\leq & \lambda_1|x^{v^1}| + \cdots + \lambda_s|x^{v^s}| \\
&\leq & |x^{v^1}| + \cdots + |x^{v^k}| \ = \ \widetilde{\mathscr{P}}(x).
\end{eqnarray*}
Hence $\frac{1}{\#(\Gamma(f) \cap \mathbb{Z}_+^n)} \mathscr{P}(x) \le \widetilde{\mathscr{P}}(x),$ which completes the proof.
}\end{remark}

The following lemma is a version at infinity of \cite[Theorem~3.2]{Bui2016}.
In the lemma, the equivalent of the statements (i) and (ii) was proved in \cite{Gindikin1974, Mikhalov1967}; 
for the sake of completeness we give a proof, which is different from the ones in these papers.

\begin{lemma} \label{Lemma34}
The following conditions are equivalent
\begin{enumerate}
        \item[(i)] $f_{\Delta} > 0$ on $(\mathbb{R}\setminus \{0\})^n$ for all $\Delta \in \Gamma_{\infty}(f);$
     	\item[(ii)] There exist positive constants $c_1, c_2,$ and $R$ such that
\begin{eqnarray} \label{Eqn1}
c_1 \mathscr{P}(x) & \le & f(x) \ \le \ c_2 \mathscr{P}(x) \quad \textrm{ for all } \quad \Vert x \Vert > R.
\end{eqnarray}  
		\item[(iii)] $f$ is non-degenerate and there exists $R > 0$ such that $f(x) \ge 0$ for all $\|x\| > R.$
\end{enumerate}
\end{lemma}

\begin{proof}
(i) $\Rightarrow$ (ii) Suppose that $f$ is written as $f = \sum_{\alpha} a_\alpha x^\alpha.$ We have for all $x \in \mathbb{R}^n,$
\begin{eqnarray*}
f(x) &\le& 
\sum_{\alpha} |a_\alpha| |x^\alpha| \ \le \  \max_{\alpha} |a_\alpha| \sum_{\alpha} |x^\alpha| 
\ \le \  \max_{\alpha} |a_\alpha| \mathscr{P}(x),
\end{eqnarray*}  
and so the right-hand inequality in~\eqref{Eqn1} holds with $c_2 := \max_{\alpha} |a_\alpha| > 0.$

Suppose the left-hand inequality in~\eqref{Eqn1} was false. By the Curve Selection Lemma at infinity (see \cite[Theorem~1.12]{HaHV2017}, \cite{Milnor1968}), then we could find analytic curves $\phi \colon (0,\epsilon) \rightarrow \mathbb{R}^n, t \mapsto (\phi_1(t),\ldots, \phi_n(t)),$ and $c \colon (0,\epsilon) \rightarrow \mathbb{R}$ such that 
\begin{enumerate}
        \item[(a)] $\Vert \phi(t) \Vert \rightarrow +\infty \textrm{ as } t \rightarrow 0^+$;
        \item[(b)] $c(t) > 0 \textrm{ for } t \in (0,\epsilon),$  $c(t) \rightarrow 0 \textrm{ as } t\rightarrow 0^+$;
     	\item[(c)] $c(t) \mathscr{P}(\phi(t)) > f(\phi(t)) \textrm{ for } t \in (0,\epsilon).$
\end{enumerate}

Let $J := \{j \ | \ \phi_j \not\equiv 0 \} \subset \{1, \ldots, n\}.$ By Condition~(a), $J \neq \emptyset.$ We can expand the functions $c(t)$ and $\phi_j(t)$ for $j \in J,$ in terms of the parameter, say
\begin{eqnarray*}
c(t) &=&  c_0 t^p + \textrm{ higher-order terms in }t \\
\phi_j(t) &=& x^0_j t^{q_j} + \textrm{ higher-order terms in }t,
\end{eqnarray*}
where $c_0 \ne 0, x^0_j \neq 0$ and $p, q_j \in \mathbb{Q}.$ By conditions (a) and (b), $c_0 > 0$ and $p > 0 > \min_{j \in J} q_j.$

Recall that $\mathbb{R}^J:= \{\alpha \in \mathbb{R}^n  \ | \  \alpha_j = 0 \textrm{ for } j \notin J\}$. If $\mathbb{R}^J \cap \Gamma(f) = \emptyset$, then for each $\alpha \in \Gamma(f)$, there exists an index $j \notin J$ such that $\alpha_j > 0.$ Consequently,
\begin{eqnarray*}
\mathscr{P}(\phi(t)) &\equiv& \sum_{\alpha \in \Gamma(f) \cap \mathbb{Z}_+^n} \vert \phi(t)^{\alpha}\vert \ \equiv \ \sum_{\alpha \in \Gamma(f) \cap \mathbb{Z}_+^n} \left ( \prod_{j \in J} \vert \phi_j(t)^{\alpha_j}\vert \prod_{j \notin J} \vert \phi_j(t)^{\alpha_j}\vert \right)  \ \equiv \ 0.
\end{eqnarray*}
Similarly, we also have $f(\phi(t)) \equiv 0,$ which contradicts Condition~(c).

Therefore, $\mathbb{R}^J \cap \Gamma(f) \neq \emptyset$. Let $d$ be the minimal value of the linear function $\sum_{j \in J}\alpha_j q_j$ on $\mathbb{R}^J\cap \Gamma(f)$ and $\Delta$ be the maximal face of $\Gamma(f)$ where this linear function takes its minimum value. Then $\Delta \in \Gamma_{\infty} (f)$ since $\min_{j \in J} q_j < 0.$
Furthermore, we have asymptotically as $t \to 0^+,$
\begin{eqnarray*}
c(t)\mathscr{P}(\phi(t)) &=&  c_0 \mathscr{P}_{\Delta}(x^0) t^{d+p} + \textrm{ higher-order terms in }t, \\
f(\phi(t)) &=& f_{\Delta}(x^0) t^d  + \textrm{ higher-order terms in }t,
\end{eqnarray*}
where $x^0 := (x^0_1,\ldots, x^0_n)$ with $x^0_j :=1$ for $ j \notin J.$ Note that $\mathscr{P}_{\Delta}(x^0) > 0$ and $f_{\Delta}(x^0) > 0.$ Therefore, by Condition~(c), we get
\begin{eqnarray*}
d + p & \le & d,
\end{eqnarray*}
which contradicts the fact that $p > 0.$

(ii) $\Rightarrow$ (iii) The left-hand inequality in~\eqref{Eqn1} shows that $f(x) \ge 0$ for all $\|x\| > R.$

Take any $x^0 \in (\mathbb{R} \setminus \{0\})^n$ and $\Delta \in \Gamma_{\infty}(f).$ By definition, there exists a vector $q \in \mathbb{R}^n$ with $\min_{j = 1, \ldots, n} q_j < 0$ such that  $\Delta = \Delta(q, \Gamma(f)).$ Consider the monomial curve 
$$\phi \colon (0, + \infty) \rightarrow \mathbb{R}^n, \quad t \mapsto (x^0_1 t^{q_1}, \ldots, x^0_n t^{q_n}).$$
Clearly, $\Vert \phi(t) \Vert\rightarrow +\infty$ as $t\rightarrow 0^+.$ Furthermore, we have asymptotically as $t \to 0^+,$
\begin{eqnarray*}
\mathscr{P}(\phi(t)) &=& \mathscr{P}_{\Delta}(x^0) t^d + \textrm{ higher-order terms in }t,  \\
f(\phi(t)) &=& f_{\Delta}(x^0) t^d + \textrm{ higher-order terms in }t,
\end{eqnarray*}
where $d := d(q, \Gamma(f)).$ Since $\mathscr{P}_{\Delta}(x^0) > 0,$ it follows from~\eqref{Eqn1} that $ f_{\Delta}(x^0)  > 0.$ In particular, $f$ is nondegenerate. 

(iii) $\Rightarrow$ (i) Take any $\Delta \in \Gamma_\infty(f).$ We first show that $f_{\Delta} \ge 0$ on $(\mathbb{R}\setminus \{0\})^n.$ On the contrary, suppose that 
$f_{\Delta}(x^0) < 0$ for some $x^0\in (\mathbb{R}\setminus \{0\})^n.$ By definition, there exists a vector $q \in \mathbb{R}^n$ with $\min_{j = 1, \ldots, n} q_j < 0$ such that  $\Delta = \Delta(q, \Gamma(f)).$ Consider the monomial curve 
$$\phi \colon (0, +\infty) \rightarrow \mathbb{R}^n, \quad t \mapsto (x^0_1 t^{q_1}, \ldots, x^0_n t^{q_n}).$$
Clearly, $\Vert \phi(t) \Vert\rightarrow +\infty$ as $t\rightarrow 0^+.$ Furthermore, we have asymptotically as $t \to 0^+,$
\begin{eqnarray*}
f(\phi(t)) &=& f_{\Delta}(x^0) t^d + \textrm{ higher-order terms in }t,
\end{eqnarray*}
where $d := d(q, \Gamma(f)).$ Since $f_{\Delta}(x^0) < 0,$ it follows that $f < 0$ on the curve $\phi,$ which contradicts our assumption.

Therefore, $f_{\Delta} \ge 0$ on $(\mathbb{R}\setminus \{0\})^n,$ and by continuity, $f_{\Delta} \ge 0$ on $\mathbb{R}^n.$

We next show that $f_{\Delta} > 0$ on $(\mathbb{R}\setminus \{0\})^n.$ By contradiction, suppose that $f_{\Delta}(x^0) = 0$ 
for some $x^0 \in (\mathbb{R}\setminus \{0\})^n.$ Since $f_{\Delta} \ge 0$ on $\mathbb{R}^n,$ this follows that $x^0$ is a global minimizer of $f_{\Delta}$ on $\mathbb{R}^n,$ and so $x^0$ is a critical point of $f_{\Delta}.$ Therefore, 
$$f_{\Delta}(x^0) =  \frac{\partial f_{\Delta}(x^0)}{\partial x_1} = \cdots = \frac{\partial f_{\Delta}(x^0)}{\partial x_n} = 0,$$
which contradicts the non-degeneracy of $f.$
\end{proof}

The following result presents necessary and sufficient conditions for the stable compactness in terms of the Newton polyhedron of the defining polynomial.

\begin{theorem}
[{Compare \cite[Theorem~5.1]{Marshall2003}}]
\label{Theorem3}
The following conditions are equivalent:
\begin{enumerate}
	\item[(i)]  $\mathcal{Z}(f)$ is stably compact.
	
	\item[(ii)]  $f|_{\mathbb{R}^J} \not\equiv 0$ for all $J \subset \{1,\ldots,n\}$ and $f_{\Delta} \ne 0$ on $(\mathbb{R}\setminus \{0\})^n$ for all $\Delta \in \Gamma_{\infty}(f)$.
	
	\item[(iii)] $f|_{\mathbb{R}^J} \not\equiv 0$ for all $J \subset \{1,\ldots,n\}$ and one of the following statements holds
	\begin{enumerate}
	\item[(iii1)] $f_{\Delta} > 0$ on $(\mathbb{R}\setminus \{0\})^n$ for all $\Delta \in \Gamma_{\infty}(f)$.
	\item[(iii2)] $f_{\Delta} < 0$ on $(\mathbb{R}\setminus \{0\})^n$ for all $\Delta \in \Gamma_{\infty}(f)$.
	\end{enumerate}

	\item[(iv)] $f|_{\mathbb{R}^J} \not\equiv 0$ for all $J \subset \{1,\ldots,n\}$ and there exist $\sigma \in \{-1, 1\}$ and constants $c_1 > 0, c_2 > 0,$ and $R > 0$ such that 
\begin{equation}\label{Eqn2}
c_1 \mathscr{P}(x) \ \le \ \sigma f(x) \ \le \ c_2 \mathscr{P}(x)  \quad \textrm{ for all } \quad \Vert x \Vert >R.
\end{equation} 

\item[(v)] $f|_{\mathbb{R}^J} \not\equiv 0$ for all $J \subset \{1,\ldots,n\},$ $f$ is non-degenerate, and there exist $\sigma \in \{-1, 1\}$ and $R > 0$ such that 
$\sigma f(x) \ge 0$ for all $\| x \| >R.$
\end{enumerate}
\end{theorem} 

\begin{proof}
(ii) $\Leftrightarrow$ (iii) $\Leftrightarrow$ (iv) $\Leftrightarrow$ (v)  follow immediately from Lemmas~\ref{Lemma33}~and~\ref{Lemma34}. Hence, it suffices to show 
(i) $\Rightarrow$ (iii) and (iv) $\Rightarrow$ (i).

(i) $\Rightarrow$ (iii) By assumption, the set $\mathcal{Z}(f )$ is compact. Thanks to Theorem~\ref{Theorem1},  $f|_{\mathbb{R}^J} \not\equiv 0$ for all $J \subset \{1,\ldots,n\},$ and 
replacing $f$ by $-f$ if necessary, we may assume that $f$ is bounded from below and $f_{\Delta} \ge 0$ on $(\mathbb{R}\setminus \{0\})^n$ for all $\Delta \in \Gamma_{\infty}(f).$
We will show (iii1) holds. On the contrary, suppose that there exist $x^0 \in (\mathbb{R}\setminus \{0\})^n$ and $\Delta \in \Gamma_{\infty}(f)$ such that $f_{\Delta}(x^0) = 0.$ This implies that $\Delta$ contains at least two vertices, say $\Delta_1$ and $\Delta_2.$ Note that all the coordinates of the vertices $\Delta_1$ and $\Delta_2$ are even integer numbers because $f$ is bounded from below. This implies easily that $f_{\Delta_1}(x^0) > 0$ and $f_{\Delta_2}(x^0) > 0.$

For each $\epsilon > 0$ consider the polynomial $g_{\epsilon}(x) := - \epsilon x^{\Delta_1}.$ Clearly, 
$g_{\epsilon}(x) = - \epsilon f_{\Delta_1}(x),$ $\Gamma(g_{\epsilon}) \subset \Gamma(f),$ and $\Gamma_{\infty}(f + g_{\epsilon}) = \Gamma_{\infty}(f)$ for all $\epsilon > 0$ small enough. Furthermore, we have
\begin{eqnarray*}
(f + g_{\epsilon})_{\Delta}(x^0) &=& f_{\Delta}(x^0)  +  g_{\epsilon, \Delta}(x^0)  \ = \ - \epsilon f_{\Delta_1}(x^0) \ < \ 0, \\
(f + g_{\epsilon})_{\Delta_2}(x^0) &=& f_{\Delta_2}(x^0) \ > \ 0.
\end{eqnarray*}
By Theorem~\ref{Theorem1}, $\mathcal{Z}(f + g_\epsilon)$ is not compact, a contradiction.

(iv) $\Rightarrow$ (i) Without loss of generality, we may assume that (iv) holds with $\sigma = 1.$  Let $f$ be written as $f = \sum_{\alpha} a_\alpha x^\alpha$ and set 
$$\epsilon := \min \left\{\frac{c_1}{2}, \min_{\alpha} |a_\alpha| \right\}> 0,$$ 
where the second minimum is taken over all the vertices of $\Gamma(f).$

Take any polynomial $g \colon \mathbb{R}^n \rightarrow \mathbb{R}$ with $\Gamma(g) \subseteq \Gamma(f)$ and $\Vert g \Vert < \epsilon.$ 
By definition, then $\Gamma_{\infty}(f + g) = \Gamma_{\infty}(f).$ Furthermore, we have for all $x \in \mathbb{R}^n,$
\begin{eqnarray*}
\vert g(x) \vert & \le & \|g\| \sum_{\alpha \in \Gamma(f)} | x^\alpha| \ \leq \ \frac{c_1}{2} \mathscr{P}(x).
\end{eqnarray*}
It follows from~\eqref{Eqn2} that
\begin{eqnarray} \label{Eqn3}
\frac{c_1}{2} \mathscr{P}(x) &\le& (f + g)(x) \ \le \ \left( \frac{c_1}{2} + c_2\right) \mathscr{P}(x)  \quad \textrm{ for } \quad \Vert x \Vert > R.
\end{eqnarray}
Consequently, we have for all $J \subset \{1,\ldots,n\},$ $(f + g)|_{\mathbb{R}^J} \not \equiv 0$ since otherwise $\mathscr{P}|_{\mathbb{R}^J} \equiv 0,$ and hence, by \eqref{Eqn2}, 
$f|_{\mathbb{R}^J} \equiv 0,$ a contradiction.

Furthermore, from \eqref{Eqn3} and Lemma~\ref{Lemma34}, we deduce that $(f + g)_\Delta > 0$ on $(\mathbb{R}\setminus \{0\})^n$ for all $\Delta \in \Gamma_{\infty}(f + g).$ 

Therefore, in view of Theorem~\ref{Theorem2}, the set $\mathcal{Z}(f + g)$ is compact. 
\end{proof}

\begin{remark}{\rm
We would like to mention that the results obtained in this paper can be used to examine the (stable) compactness of basic closed semi-algebraic sets. To see this, let $X$ be a basic closed semi-algebraic set defined by
$$X := \{x \in {\Bbb R}^n \ | \ g_1(x) = 0, \ldots,  g_l(x)  = 0, \ h_1(x) \ge 0, \ldots, h_m(x) \ge 0 \},$$
where $g_1, \ldots, g_l, h_1, \ldots, h_m$ are polynomial functions on $\mathbb{R}^n.$ It is easy to see that $X$ is compact if, and only if, the set
$$Y := \{(x, y) \in {\Bbb R}^n \times {\Bbb R}^m  \ | \ g_1(x) = 0, \ldots,  g_l(x)  = 0, \ h_1(x) - y_1^2 = 0, \ldots, h_m(x) - y_m^2 = 0 \},$$
is compact. Then the statement follows because $Y$ is the zero set of the polynomial function
$${\Bbb R}^n \times {\Bbb R}^m \rightarrow \mathbb{R}, (x, y) \mapsto [g_1(x)]^2 + \cdots + [g_l(x)]^2 + [h_1(x) - y_1^2]^2 + \cdots + [h_m(x) - y_m^2]^2.$$
}\end{remark}

\subsection*{Acknowledgments}
The authors thank S\~i-Ti\d{\^e}p \DD inh for useful discussions during the preparation of this paper.
The final version of this paper was completed while the second author was visiting LAAS--CNRS in April 2017. The second author wishes to thank the institute and  Jean Bernard Lasserre in particular, for the hospitality and financial support from the European Research Council (ERC) through the ERC-Advanced Grant: TAMING 666981

\end{document}